\documentclass[11pt,oneside,reqno]{amsart}
\usepackage{amssymb,amsmath,amsthm}
\usepackage{a4wide,comment,enumitem,url}
\usepackage{xcolor}
\usepackage{enumitem}
\usepackage{bm}
\usepackage{float}
\usepackage{bbm}

\newtheorem{theorem}{Theorem}

\newtheorem{proposition}[theorem]{Proposition}

\setlist[enumerate]{leftmargin=*}
\setlist[itemize]{leftmargin=*}

\theoremstyle{remark}

\newtheorem*{remark}{\bf Remark}

\numberwithin{theorem}{section} \numberwithin{equation}{section}

\usepackage{mathtools}

\theoremstyle{plain}

\theoremstyle{definition}

\usepackage{enumitem}
\setlist[itemize]{noitemsep, topsep=0pt}
\setlist[enumerate]{noitemsep, topsep=0pt}

\allowdisplaybreaks

\title{Distribution of the sum of reciprocal parts for distinct parts partitions}

\author{Walter Bridges}
\address{University of North Texas, Department of Mathematics, Denton, TX, USA}
\email{Walter.Bridges@unt.edu}

\date{\today}

\subjclass[2020]{11P82, 11D68, 60C05}
\keywords{egyptian fractions, distinct parts partitions, distribution of partition statistics}

\begin{document}
\maketitle

\begin{abstract}
    Given an integer partition of $n$ into distinct parts, the sum of the reciprocal parts is an example of an egyptian fraction.  We study this statistic under the uniform measure on distinct parts partitions of $n$ and prove that, as $n \to \infty$, the sum of reciprocal parts is distributed away from its mean like a random harmonic sum.
\end{abstract}

\section{Introduction}

A {\it distinct parts partition} of $n \in \mathbb{N}$ is a sequence of integers,
\begin{equation}\label{E:DistinctPartitionDef}
\lambda_1 > \dots > \lambda_{\ell} >0, \qquad \text{with} \qquad \sum_{j=1}^{\ell}\lambda_j=n.
\end{equation}
In this paper, we study the sum of reciprocal parts,
$$
S(\lambda):=\sum_{j=1}^n \frac{1}{\lambda_j}.
$$
With respect to the uniform probability measure on distinct parts partitions of $n$, we prove that, away from its mean, $S$ is distributed like a random harmonic sum as $n \to \infty$.  Let
$$
H:=\sum_{k \geq 1} \frac{\varepsilon_k}{k},
$$
where the $\varepsilon_k$ are independent random variables with\footnote{Unless stated otherwise, $P$ denotes the probability measure induced by the random variable in its argument.} $P(\varepsilon_k=\pm 1)=\frac{1}{2}$.  The random harmonic sum $H$ converges almost surely, for example by summation by parts and the law of the iterated logarithm; for further properties see \cite{Schmuland}.  Let $P_n$ be the uniform measure on distinct parts partitions of $n$.  
\begin{theorem}\label{T:Main}
For any $x \in \mathbb{R}$, we have
$$
\lim_{n \to \infty} P_n\left(2S-\log(\sqrt{3n})\leq x \right) = P(H\leq x).
$$
\end{theorem}

\begin{figure}[h] 
\centering
\includegraphics[scale=.4]{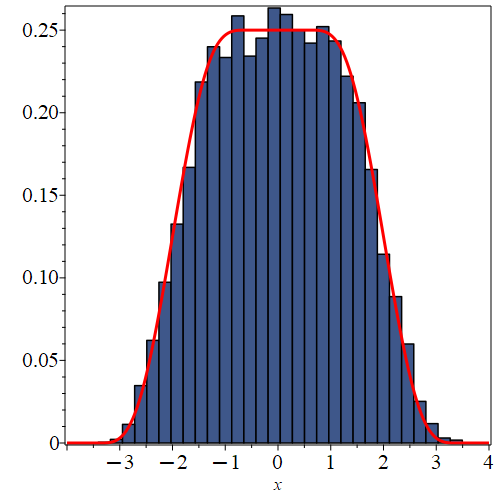}    
\caption{A histogram of 10\hspace{.7mm}000 values of $2S(\lambda)-\log(\sqrt{3|\lambda|})$, where partitions $\lambda$ have been generated in Maple by a Boltzmann sampler with parameter $q=e^{-\frac{\pi}{\sqrt{12n}}}$ with $n=2000$ (see \cite{DFLS,Fristedt}).  In red is an approximation to the density for $H$, as described in \cite[\S 5]{Schmuland}.}
\label{Histogram}
\end{figure}

Egyptian fractions - that is, sums of unit fractions with distinct denominators - have been studied since antiquity and remain a fascinating topic today with numerous open problems; see \cite[Ch. 3]{EG} and \cite[\S D11]{Guy}.  For instance, a well-known result of Graham \cite{Graham} states that every $n \geq 78$ has a distinct parts partition $\lambda$ with $S(\lambda)=1$. Much of the literature has focused on fixing the value of the sum $S$ or the number of the summands, rather than the size of the partition as we do here (for example, see \cite{Conlon,vD}).  Theorem \ref{T:Main} complements these results by addressing the distribution of egyptian fractions under natural order and size restrictions placed on denominators by \eqref{E:DistinctPartitionDef}.

The sum of reciprocal parts of distinct parts partitions was recently studied by B. Kim and E. Kim \cite{KK}, where the authors used the saddle-point method to prove asymptotic expansions for the first and second moments.  Bringmann, B. Kim and E. Kim then significantly improved these asymptotic expansions from exponential to square-root error utilizing modular transformations and the Hardy--Ramanujan Circle Method \cite{BKK}. 
 In the case that parts of a partition are allowed to repeat, the distribution of $S$ may be found in \cite{KK} and follows in a straightforward manner from known distributions of partition statistics.  On the other hand, when parts are distinct, the computation of the distribution is apparently more subtle.

Our approach to Theorem \ref{T:Main} is entirely elementary and combines the distribution of small part sizes due to Fristedt \cite{Fristedt} with a strong version of the limit shape for distinct parts partitions \cite{DVZ,Petrov,Yakubovich}.  The limit shape is re-derived in a form suitable for our purposes in the Appendix.

To describe our approach to Theorem \ref{T:Main}, let $X_k \in \{0,1\}$ be the multiplicity of the part size $k$.  Thus, we can study $S$ as the random variable,
\begin{equation}\label{E:Ssetup}
S=\sum_{k = 1}^n \frac{X_k}{k}, \qquad \text{subject to} \qquad \sum_{k = 1}^n kX_k = n.
\end{equation}
We break up the sum \eqref{E:Ssetup} into three ranges,
$$
[1,n]=[1,k_n] \cup (k_n,K_n] \cup (K_n,n],
$$ where\footnote{We have fixed these choices simply for convenience, and it is likely that the errors in Sections \ref{S:MediumParts} and \ref{S:LargeParts} could be improved.} $k_n:=\lfloor n^{1/5} \rfloor$ and $K_n:= \lfloor n^{1/3} \rfloor$.  By \cite[Theorem 9.2]{Fristedt}, the joint distribution of the $X_k$ in the range $k \leq k_n$ coincides in the limit with the joint distribution of $k_n$ independent Bernoulli random variables with parameter $\frac{1}{2}$, and, after renormalizing, this immediately leads to a random harmonic sum.  Next, we will show that the sums over the ranges $k_n<k \leq K_n$ and $K_n<k\leq n$ are almost surely $o(1)$ away from their mean.  The range $k_n<k \leq K_n$ will be dealt with using \cite[\S 10]{Fristedt} and Chebyshev's inequality; our overall argument is very similar to \cite[\S 8]{Fristedt}.  For the range $K_n < k \leq n$, we will employ the limit shape for distinct parts partitions in a suitable form \cite{DVZ,Petrov,Yakubovich}.

The remainder is organized as follows.  We study the contributions to the sum of reciprocal parts from small, intermediate and large part sizes in Sections \ref{S:SmallParts}, \ref{S:MediumParts} and \ref{S:LargeParts}, respectively.  In Section \ref{S:LargeParts}, we make use of suitable form of the limit shape of distinct parts partitions that we derive in the Appendix.  We complete the proof of Theorem \ref{T:Main} in Section \ref{S:Proof}.

\section{Contribution from small parts} \label{S:SmallParts}

The following proposition indicates that the sum of reciprocals of small parts is distributed, away from its mean, like the random harmonic sum.  Here and throughout, $\gamma=.577...$ is the Euler--Mascheroni constant.

\begin{proposition}\label{P:SmallParts}
For any $x \in \mathbb{R}$, we have
$$
\lim_{n \to \infty} P_n\left(\sum_{k \leq k_n} \frac{2X_k}{k}-\log\left(k_n\right) - \gamma\leq x \right) = P(H\leq x).
$$
\end{proposition}

\begin{proof}
    By the discussion following \cite[Theorem 9.2]{Fristedt}, the density of the tuple $(2X_1-1,\dots, 2X_{k_n}-1)$ is asymptotic to the density of $(\varepsilon_1,\dots, \varepsilon_{k_n})$ as $n \to \infty$.  That is, for any sequence $\{x_k\}_{k \geq 1}$, where $x_k \in \{\pm 1\}$, we have
    $$
   P_n(X_k=x_k, \ k \leq k_n)\sim P(\varepsilon_k=x_k, \ k \leq k_n) = \frac{1}{2^{k_n}}.
    $$
The well-known expansion of the harmonic sum gives
$$
\sum_{k \leq k_n} \frac{2X_k}{k}-\log k_n-\gamma=\sum_{k \leq k_n} \frac{2X_k-1}{k} + o(1),
$$
so an immediate consequence is that
$$
\lim_{n \to \infty} \left(P_n\left(\sum_{k \leq k_n} \frac{2X_k}{k}-\log\left(k_n\right) - \gamma\leq x \right) - P\left( \sum_{k\leq k_n} \frac{\varepsilon_k}{k} \leq x\right)\right) = 0.
$$
Since the partial sums $\sum_{k \leq k_n}\frac{\varepsilon_k}{k}$ converge in distribution to $H$, this proves the proposition.
\end{proof}

\section{Contribution from intermediate parts} \label{S:MediumParts}

The following proposition shows that the contribution from intermediate parts is very small away from the mean.

\begin{proposition}\label{P:MediumParts}
We have
    $$
\lim_{n \to \infty} P_n\left(\left|\sum_{k_n<k\leq K_n} \frac{2X_k}{k}-\log\left(\frac{K_n}{k_n}\right)\right|\leq n^{-\frac{1}{11}} \right) = 1.
$$
\end{proposition}

\begin{proof}
  We apply \cite[Lemma 4.6]{Fristedt}, adjusted for distinct parts partitions as in the discussion in \cite[\S 10]{Fristedt}.  The condition we must verify is as follows: set $q_n=e^{-\frac{1}{A\sqrt{n}}}$ with $A=\frac{\sqrt{12}}{\pi}$.  By integral comparison, we have
  \begin{align*}
  \sum_{k_n<k\leq K_n} \frac{k^2q_n^k}{(1+q_n^k)^2} &= A^3n^{\frac{3}{2}}\sum_{k_n<k\leq K_n} \frac{q_n^k}{(1+q_n^k)^2}\left(\frac{k}{A\sqrt{n}}\right)^2\frac{1}{A\sqrt{n}} \\
  &\ll n\int_{\frac{k_n}{A\sqrt{n}}}^{\frac{K_n}{A\sqrt{n}}} \frac{ue^{-u}}{(1+e^{-u})^2}du \\
  &\ll nK_n \\
  &= o\left(n^{\frac{3}{2}}\right),
 \end{align*}
  so that \cite[Lemma 4.6]{Fristedt} applies.  It follows that
 \begin{multline}
  \lim_{n \to \infty} \left(P_n\left(\left|\sum_{k_n<k\leq K_n} \frac{2X_k}{k}-\log\left(\frac{K_n}{k_n}\right)\right|\leq n^{-\frac{1}{11}} \right) \right. \\ \left. - P\left(\left|\sum_{k_n<k\leq K_n} \frac{2B_k}{k}-\log\left(\frac{K_n}{k_n}\right)\right|\leq n^{-\frac{1}{11}} \right)\right) = 0, \label{E:TotalVariation}
 \end{multline}
  where the $B_k$ are independent Bernoulli random variables with parameter $1+q_n^k$.  We prove that the second term tends to 0 by Chebyshev's inequality.
  
  The expectation and variance of the sum of $\frac{2B_k}{k}$ can be approximated by integral comparison:
\begin{align*}
E\left(\sum_{k_n<k\leq K_n} \frac{2B_k}{k}\right)&= \sum_{k_n<k\leq K_n}\frac{2q_n^k}{k(1+q_n^k)} \\
&= \sum_{k_n<k\leq K_n}\frac{2q_n^k}{1+q_n^k} \left(\frac{k}{A\sqrt{n}}\right)^{-1}\frac{1}{A\sqrt{n}}\\
&=\int_{\frac{k_n}{A\sqrt{n}}}^{\frac{K_n}{A\sqrt{n}}} \frac{2e^{-u}}{u(1+e^{-u})}du + O\left( \left(\frac{k_n}{\sqrt{n}}\right)^{-1}\frac{e^{-\frac{k_n}{A\sqrt{n}}}}{1+e^{-\frac{k_n}{A\sqrt{n}}}}\frac{1}{\sqrt{n}} \right) \nonumber \\
&=\log\left(\frac{K_n}{k_n}\right)+\int_{\frac{k_n}{A\sqrt{n}}}^{\frac{K_n}{A\sqrt{n}}} \left(\frac{2e^{-u}}{u(1+e^{-u})}-\frac{1}{u}\right)du + O\left( \frac{1}{k_n} \right) \nonumber \\
    &=\log\left(\frac{K_n}{k_n}\right) +O\left(\frac{K_n}{\sqrt{n}}+\frac{1}{k_n}\right) \nonumber \\
    &=\log\left(\frac{K_n}{k_n}\right) +O\left(n^{-\frac{1}{6}}\right). \nonumber
\end{align*}
And, using independence, we have
\begin{align*}
\mathrm{Var}\left(\sum_{k_n<k\leq K_n} \frac{2B_k}{k}\right) &=\sum_{k_n<k\leq K_n}\frac{4}{k^2}\mathrm{Var}\left( B_k\right)\\
&= \sum_{k_n<k\leq K_n}\frac{4q_n^k}{k^2(1+q_n^k)^2}  \\
&= \frac{1}{A\sqrt{n}}\sum_{k_n<k\leq K_n}\frac{4q_n^k}{(1+q_n^k)^2} \left(\frac{k}{A\sqrt{n}}\right)^{-2}\frac{1}{A\sqrt{n}} \\
&=\frac{1}{A\sqrt{n}}\int_{\frac{k_n}{A\sqrt{n}}}^{\frac{K_n}{A\sqrt{n}}} \frac{4e^{-u}}{u^2(1+e^{-u})^2}du + O\left(\frac{1}{\sqrt{n}}\frac{e^{-\frac{k_n}{A\sqrt{n}}}}{1+e^{-\frac{k_n}{A\sqrt{n}}}} \left(\frac{k_n}{\sqrt{n}}\right)^{-2}\frac{1}{\sqrt{n}}\right) \nonumber \\
    &=O\left(\frac{1}{\sqrt{n}} \cdot \frac{\sqrt{n}}{k_n} +\frac{1}{k_n^2}\right) \\
    &=O\left(n^{-\frac{1}{5}}\right). \nonumber
\end{align*}
By Chebyshev's inequality, we conclude
$$
P\left(\left|\sum_{k_n<k\leq K_n} \frac{2B_k}{k} - \log\left(\frac{K_n}{k_n}\right)\right|>n^{-\frac{1}{11}} \right) \ll \frac{n^{\frac{2}{11}}}{ n^{\frac{1}{5}}}=o(1).
$$
Combining with \eqref{E:TotalVariation}, the Proposition follows.
\end{proof}

\begin{remark}
    The proof of Proposition \ref{P:MediumParts} is similar to the arguments in \cite[\S 8]{Fristedt}.  Namely, we would like to extend Proposition \ref{P:SmallParts} from the range $k \leq k_n$ to $k \leq K_n$, but it is not necessarily true that the joint distribution of the $(2X_k-1)$ coincides in the limit with the joint distribution of the $\varepsilon_k$.  Indeed, \cite[Theorem 9.2]{Fristedt} only applies in the range $k=o\left(n^{\frac{1}{4}}\right)$ and $K_n >n^{\frac{1}{4}}$.  However, the {\it sum} of $\frac{2X_k-1}{k}$ has small enough variance so that its distribution does in fact coincide with the sum of $\frac{\varepsilon_k}{k}$ for $k \leq K_n$.
\end{remark}

\section{Contribution from large parts} \label{S:LargeParts}

The following proposition shows that the contribution from the large parts is very small away from the mean.

\begin{proposition}\label{P:LargeParts}
We have
    $$
\lim_{n \to \infty} P_n\left(\left|\sum_{K_n<k\leq n} \frac{2X_k}{k}-\log\left(\frac{\sqrt{3n}}{K_n} \right) + \gamma\right|\leq n^{-\frac{1}{30}} \right) = 1.
$$
\end{proposition}

\begin{proof}
      Set $\Delta_n:=n^{-\frac{1}{5}}$ and for $j =0,1,\dots, J$, let
    $$
    t_{j,n}:=\frac{K_n}{\sqrt{n}}+j\Delta_n,
    $$
    where $J=J_n$ is taken to satisfy
   $
   t_{J,n} \leq \sqrt{n} < t_{J+1,n}.
   $
    Observe that  
$$
\frac{1}{t_{j,n}\sqrt{n}} \sum_{t_{j-1,n}\sqrt{n}<k\leq t_{j,n}\sqrt{n}} X_k\leq \sum_{t_{j-1,n}\sqrt{n}<k\leq t_{j,n}\sqrt{n}}\frac{X_k}{k} \leq \frac{1}{t_{j-1,n}\sqrt{n}} \sum_{t_{j-1,n}\sqrt{n}<k\leq t_{j,n}\sqrt{n}} X_k.
$$
Thus, Proposition \ref{P:LimitShape} implies, for any $0<\delta<\frac{1}{4}$,
\begin{multline}
\sum_{j=1}^{J+1}\frac{2}{t_{j,n}} (L(t_{j,n})-L(t_{j-1,n})) + O\left(n^{-\frac{1}{4}+\delta}\sum_{j=1}^{J+1} \frac{1}{t_{j,n}}\right)\leq \sum_{K_n<k\leq n}\frac{2X_k}{k} \\ \leq \sum_{j=1}^{J+1}\frac{2}{t_{j-1,n}} (L(t_{j,n})-L(t_{j-1,n})) + O\left(n^{-\frac{1}{4}+\delta}\sum_{j=1}^{J+1} \frac{1}{t_{j-1,n}}\right), \label{E:FirstIneq}
\end{multline}
almost surely.  By integral comparison, we have 
   $$
   0 \leq \sum_{j=1}^{J+1} \frac{1}{t_{j-1,n}}\Delta_n - \int_{t_{0,n}}^{t_{J+1,n}} \frac{1}{t}dt  \leq \frac{\Delta_n}{t_{0,n}}\leq n^{-\frac{1}{30}},
   $$
   which implies
   $$
   n^{-\frac{1}{4}+\delta}\sum_{j=1}^{J+1} \frac{1}{t_{j-1,n}}\ll \frac{n^{-\frac{1}{4}+\delta}}{\Delta_n}\log n= n^{-\frac{1}{20}+\delta}\log n.
   $$
   Thus, \eqref{E:FirstIneq} becomes
   \begin{multline}
\sum_{j=1}^{J+1}\frac{2}{t_{j,n}} (L(t_{j,n})-L(t_{j-1,n})) + O\left(n^{-\frac{1}{20}+\delta}\log n\right)\leq \sum_{K_n<k\leq n}\frac{2X_k}{k} \\ \leq \sum_{j=1}^{J+1}\frac{2}{t_{j-1,n}} (L(t_{j,n})-L(t_{j-1,n})) + O\left(n^{-\frac{1}{20}+\delta}\log n\right), \label{E:SecondIneq}
\end{multline}
   almost surely.  For the sums on $L$, we first apply Taylor's Theorem; using the fact that $L''(t)$ is bounded, we have
$$
L(t_{j,n})-L(t_{j-1,n})=\Delta_nL'(t_{j,n}) + O(\Delta_n^2).
$$
Hence,
\begin{align}
\sum_{j=1}^{J+1} \frac{2}{t_{j,n}}(L(t_{j,n})-L(t_{j-1,n})) &= \sum_{j=1}^{J+1} \frac{2L'(t_{j,n})}{t_{j,n}}\Delta_n + O\left(\Delta_n\sum_{j=1}^{J+1} \frac{1}{t_{j,n}}\Delta_n \right) \nonumber \\ &= \sum_{j=1}^{J+1} \frac{2L'(t_{j,n})}{t_{j,n}}\Delta_n + O\left(n^{-\frac{1}{5}}\log n \right).\label{E:Ldifferencesum}
\end{align}
For the main term, we use integral comparison once again to write
   $$
   0\leq \sum_{j=1}^{J+1} \frac{2L'(t_{j,n})}{t_{j,n}}\Delta_n-\int_{t_{1,n}}^{t_{J+1,n}} \frac{2L'(t)}{t}dt \leq \Delta_n\frac{2L'(t_1)}{t_1}.
   $$
   Since $L'(t)= \frac{e^{-\frac{t}{A}}}{1+e^{-\frac{t}{A}}}$ is bounded and exponentially decaying at infinity, we have
   \begin{align*}
\sum_{j=1}^{J+1} \frac{2L'(t_{j,n})}{t_{j,n}}\Delta_n &=\int_{t_{1,n}}^{t_{J+1,n}} \frac{2L'(t)}{t}dt + O\left(\frac{\Delta_n}{t_1}\right) \\
&=\int_{t_{0,n}}^{\infty} \frac{2L'(t)}{t}dt + O\left(\frac{\Delta_n}{t_0}\right) \\
&=\int_{\frac{t_{0,n}}{A}}^{\infty} \frac{2e^{-t}}{t(1+e^{-t})}dt + O\left(n^{-\frac{1}{30}}\right) \\
&=\log\left(\frac{A}{t_0}\right)+\int_{0}^{1} \frac{1}{t}\left(\frac{2e^{-t}}{1+e^{-t}}-1\right)dt + \int_{1}^{\infty} \frac{2e^{-t}}{t(1+e^{-t})}dt + O\left(n^{-\frac{1}{30}}\right).
   \end{align*}
We evaluate the constant above using Mellin transforms as follows.  Set
$$
f(s):=\int_{0}^{1} \left(\frac{2e^{-t}}{1+e^{-t}}-1\right)t^{s-1}dt + \int_{1}^{\infty} \frac{2e^{-t}}{1+e^{-t}} t^{s-1}dt \qquad \mathrm{Re}(s)>-1,
$$
so we are interested in the value $f(0)$.  For $s>0$, a short computation reveals
$$
f(s)=-\frac{1}{s}+\int_0^{\infty} \frac{2e^{-t}}{1+e^{-t}}t^{s-1}dt= -\frac{1}{s}+2\Gamma(s)\zeta(s)\left(1-2^{1-s}\right).
$$
After computing the Taylor expansion on the right-hand side at $s=0$, we conclude $f(0)=\log\left(\frac{\pi}{2}\right)-\gamma$.  Thus,
$$
\sum_{j=1}^{J+1} \frac{2L'(t_{j,n})}{t_{j,n}}\Delta_n =\log\left(\frac{A\pi}{2t_0}\right)-\gamma+O\left(n^{-\frac{1}{30}}\right) = \log\left(\frac{\sqrt{3n}}{K_n}\right)-\gamma+O\left(n^{-\frac{1}{30}}\right).
$$
Combining with \eqref{E:Ldifferencesum} and taking $\delta<\frac{1}{20}-\frac{1}{30}$, we have
$$
\sum_{j=1}^{J+1} \frac{2}{t_{j,n}}(L(t_{j,n})-L(t_{j-1,n}))=\log\left(\frac{\sqrt{3n}}{K_n}\right)-\gamma+O\left(n^{-\frac{1}{30}}\right).
$$
Similarly, we can prove the same estimate for the right-hand side of \eqref{E:SecondIneq}.  The proposition follows.
\end{proof}

\section{Proof of Theorem \ref{T:Main}}\label{S:Proof}

We first write
\begin{multline*}
2S-\log(\sqrt{3n})=\left(\sum_{k \leq k_n} \frac{2X_k}{k}-\log\left(k_n\right) - \gamma\right) + \left(\sum_{k_n<k \leq K_n} \frac{2X_k}{k} - \log\left(\frac{K_n}{k_n}\right) \right) \\ + \left(\sum_{K_n<k \leq n} \frac{2X_k}{k} - \log\left(\frac{\sqrt{3n}}{K_n}  \right)+\gamma \right).
\end{multline*}
By Propositions \ref{P:MediumParts} and \ref{P:LargeParts}, the two right-most terms above lie in $[-n^{-1/30},n^{-1/30}]$ almost surely as $n \to \infty$.  Thus, for any fixed $x \in \mathbb{R}$, 
$$
\lim_{n \to \infty} \left(P_n\left(2S-\log(\sqrt{3n})\leq x \right) - P_n\left(\sum_{k \leq k_n} \frac{2X_k}{k}-\log\left(k_n\right) - \gamma\leq x \right)\right)=0.
$$
Theorem \ref{T:Main} now follows from Proposition \ref{P:SmallParts}.

\appendix 

\section{Strong limit shape}

Given a partition $\lambda$, the function,
$$
\phi_{\lambda}(t):= \sum_{k \leq t} X_k(\lambda),
$$
describes the shape of the Young diagram of $\lambda$ when parts are left-justified along the $y$-axis and arranged in increasing order.  If both axes are rescaled by $\frac{1}{\sqrt{n}}$ for partitions of $n$, then these random step functions converge in probability to the limit shape,
$$
L(t):=A\log\left(\frac{2}{1+e^{-\frac{t}{A}}}\right)=\int_0^{t} \frac{e^{-\frac{u}{A}}}{1+e^{-\frac{u}{A}}}du.
$$
This phenomenon is well-studied for integer partitions (for example, see \cite{DVZ,DeSalvoPak,Petrov,VershikYakubovich,Yakubovich}).  In the case of distinct parts partitions, Yakubovich \cite{Yakubovich} showed that for any fixed $0<t_1 < \dots < t_r$, the vector
$$
\frac{1}{\sqrt{n}}\left(\phi(t_1\sqrt{n}), \dots, \phi(t_r\sqrt{n}) \right) 
$$
varies from $(L(t_1),\dots,L(t_r))$ like a $r$-dimensional Gaussian at the scaling $n^{-\frac{1}{4}}$.  We prove a limit shape at a weaker scaling that covers all $t \geq 0$ simultaneously.  We closely follow the elementary proof of the limit shape given by Petrov \cite{Petrov}.

\begin{figure}[h] 
\centering
\includegraphics[scale=.4]{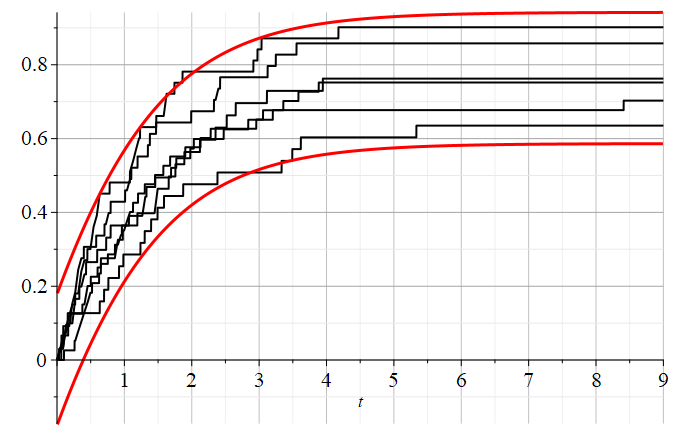}   
\caption{The black step functions are the renormalized shapes $\frac{1}{\sqrt{|\lambda|}}\phi_{\lambda}(\sqrt{|\lambda|}t)$ for six random distinct parts partitions $\lambda$ of sizes 992, 1592, 1065, 1475, 910, and 1107, generated using a Boltzmann sampler with parameter $q=e^{-\frac{\pi}{\sqrt{12n}}}$ with $n=1000$ (see \cite{DFLS,Fristedt}).  In red are the curves $L(t) \pm n^{-\frac{1}{4}}$.}
\label{HookZeros5mod7}
\end{figure}

\begin{proposition}\label{P:LimitShape}
    For fixed $0<\delta<\frac{1}{4}$, we have
    $$
    \limsup_{n \to \infty} n^{-\delta}\log P_n\left(\inf_{t \geq 0}\left|\frac{1}{\sqrt{n}}\sum_{k \leq t\sqrt{n}}X_k- L(t)\right|> n^{-\frac{1}{4}+\delta} \right) <0.
    $$
\end{proposition}

\begin{remark}
In particular, we have
    $$
    \lim_{n \to \infty} P_n\left(\sup_{t \geq 0}\left|\frac{1}{\sqrt{n}}\sum_{k \leq t\sqrt{n}}X_k- L(t)\right|\leq n^{-\frac{1}{4}+\delta} \right)=1.
    $$
\end{remark}
\begin{proof}
Let $d(n)$ be the number of distinct parts partitions of $n$.  We require only a weak form of the well-known asymptotic expansion of $d(n)$ \cite{Hagis},
$$
d(n)=e^{\frac{2\sqrt{n}}{A}+O(\log n)}.
$$

    Let $a_n,b_n \in \mathbb{N}_0$ and define $\alpha_n,\beta_n$ by
    $$
    a_n=\alpha_n\sqrt{n}, \qquad b_n=\beta_n\sqrt{n}.
    $$
    Assume that $\alpha_n,\beta_n \geq n^{-\frac{1}{4}+\delta}$.  We use the saddle point bound to write, for any $x_n \in \mathbb{R}$,
    \begin{align*}
        &P_n\left(\frac{1}{\sqrt{n}}\sum_{k \leq a_n}X_k=b_n\right) \\
        &=\frac{1}{d(n)} [q^n][\zeta^{b_n}] \prod_{k \leq a_n}(1+\zeta q^k)\prod_{k > a_n}(1+q^k) \\
        &\leq \frac{1}{d(n)}q_n^{-n}e^{-b_nx_n}\prod_{k \leq a_n}(1+e^{x_n} q_n^k)\prod_{k > a_n}(1+q_n^k) \\
        &=\exp\left(\frac{\sqrt{n}}{A} - \log d(n) - \beta_nx_n\sqrt{n}+\sum_{k \leq a_n} \log\left(1+e^{x_n-\frac{k}{A\sqrt{n}}}\right)+\sum_{k > a_n} \log\left(1+e^{-\frac{k}{A\sqrt{n}}}\right) \right).
    \end{align*}
    We will take $x_n \in \{\pm n^{-\frac{1}{4}}\}$, where the sign will depend on $b_n$.  By Taylor's Theorem,
   \begin{align*}
    \left|\log\left(1+e^{x_n-\frac{k}{A\sqrt{n}}}\right) - \log\left(1+e^{-\frac{k}{A\sqrt{n}}}\right) - x_n\frac{e^{-\frac{k}{A\sqrt{n}}}}{1+e^{-\frac{k}{A\sqrt{n}}}} \right| &\leq \frac{x_n^2}{2}\sup_{|x|\leq n^{\frac{1}{4}}} \frac{e^{x-\frac{k}{A\sqrt{n}}}}{\left(1+e^{x-\frac{k}{A\sqrt{n}}}\right)^2} \\ &\leq x_n^2 \frac{e^{-\frac{k}{A\sqrt{n}}}}{\left(1+\frac{1}{2}e^{-\frac{k}{A\sqrt{n}}}\right)^2}.
  \end{align*}
  Thus,
  \begin{multline*}
      P_n\left(\frac{1}{\sqrt{n}}\sum_{k \leq a_n}X_k=b_n\right) \leq \exp\left(\frac{\sqrt{n}}{A} - \log d(n) + \sum_{k \geq 1} \log\left(1+e^{-\frac{k}{A\sqrt{n}}}\right)- \beta_nx_n\sqrt{n} \right.  \\ \left. +x_n\sum_{k \leq a_n}\frac{e^{-\frac{k}{A\sqrt{n}}}}{1+e^{-\frac{k}{A\sqrt{n}}}} + O\left( x_n^2 \sum_{k \leq a_n}\frac{e^{-\frac{k}{A\sqrt{n}}}}{\left(1+\frac{1}{2}e^{-\frac{k}{A\sqrt{n}}}\right)^2}\right) \right). 
  \end{multline*}
  Now, note that the functions,
$$  
      u\mapsto \begin{cases} \log(1+e^{-u}) \\ \frac{e^{-u}}{1+e^{-u}} \\ \frac{e^{-u}}{(1+e^{-u})^2}, \end{cases}
$$
are decreasing, continuous and integrable on $[0,\infty)$.  Hence, by integral comparison,
\begin{align}
    \sum_{k \geq 1} \log\left(1+e^{-\frac{k}{A\sqrt{n}}}\right)&=A\sqrt{n}\sum_{k \geq 1} \log\left(1+e^{-\frac{k}{A\sqrt{n}}}\right)\frac{1}{A\sqrt{n}} \nonumber  \\
    &=A\sqrt{n}\int_0^{\infty} \log(1+e^{-u})du + O(1) \nonumber  \\
    &=\frac{\sqrt{n}}{A} + O(1), \label{E:x0term}
\end{align}
and
\begin{align}
    \sum_{k \leq a_n}\frac{e^{-\frac{k}{A\sqrt{n}}}}{1+e^{-\frac{k}{A\sqrt{n}}}}&=\sqrt{n}\sum_{k \leq a_n}\frac{e^{-\frac{k}{A\sqrt{n}}}}{1+e^{-\frac{k}{A\sqrt{n}}}}\frac{1}{\sqrt{n}} \nonumber  \\
    &=\sqrt{n}\int_0^{\frac{a_n}{\sqrt{n}}} \frac{e^{-\frac{u}{A}}}{1+e^{-\frac{u}{A}}}du + O(1) \nonumber  \\
    &=\sqrt{n}L\left(\frac{a_n}{\sqrt{n}}\right) + O(1) \nonumber  \\
    &=\sqrt{n}L(\alpha_n)+O(1), \label{E:x1term}
\end{align}
and 
\begin{align}
    \sum_{k \leq a_n}\frac{e^{-\frac{k}{A\sqrt{n}}}}{\left(1+e^{-\frac{k}{A\sqrt{n}}}\right)^2}&=A\sqrt{n}\sum_{k \leq a_n}\frac{e^{-\frac{k}{A\sqrt{n}}}}{\left(1+e^{-\frac{k}{A\sqrt{n}}}\right)}\frac{1}{A\sqrt{n}} \nonumber  \\
    &=O(\sqrt{n}), \label{E:x2term}
\end{align}
where the error terms in \eqref{E:x0term}, \eqref{E:x1term} and \eqref{E:x2term} are independent of $a_n$ and $b_n$.  Thus, using $\log d(n) = \frac{2\sqrt{n}}{A} + O(\log n)$, we have
\begin{align*}
      P_n\left(\frac{1}{\sqrt{n}}\sum_{k \leq a_n}X_k=b_n\right) &\leq \exp\left( \sqrt{n}x_n\left(-\beta_n+L(\alpha_n) + O(x_n)\right) + O(\log n) \right) \\
      &\leq \exp\left( \pm n^{\frac{1}{4}}\left(-\beta_n+L(\alpha_n)\right) + O(\log n) \right),
\end{align*}
where the sign above will depend on $\alpha_n$ and $\beta_n$ and the error is independent of $\alpha_n$ and $\beta_n$.  It follows that if
\begin{equation}\label{E:LimitShapeCondition}
\left|-\beta_n+L(\alpha_n)\right|> n^{-\frac{1}{4}+\delta},
\end{equation}
then the sign of $x_n$ may be chosen so that, uniformly for $\alpha_n$ and $\beta_n$ satisfying \eqref{E:LimitShapeCondition}, we have
$$
P_n\left(\frac{1}{\sqrt{n}}\sum_{k \leq a_n}X_k=b_n\right) \leq e^{-n^{\delta}+O(\log n)}.
$$
Since the condition \eqref{E:LimitShapeCondition} is equivalent to
$$
\left|\frac{1}{\sqrt{n}}\sum_{k \leq a_n}X_k - L(\alpha_n)\right|>n^{-\frac{1}{4}+\delta},
$$
and since there are at most $n^2$ possible pairs $(a_n,b_n)$, we conclude
\begin{equation}\label{E:LimitShapealpha}
    \limsup_{n \to \infty} n^{-\delta}\log P_n\left(\inf_{ \alpha_n \in n^{-\frac{1}{2}} \mathbb{N}_0, \alpha_n \leq \sqrt{n}}\left|\frac{1}{\sqrt{n}}\sum_{k \leq \alpha_n\sqrt{n}}X_k- L(\alpha_n)\right|> n^{-\frac{1}{4}+\delta} \right) <0.
 \end{equation}
To extend the infimum over all real $t \geq 0$, note that for $a_n\leq t\sqrt{n} <a_n+1$ with $a_n \leq n-1$, we have
$$
\frac{1}{\sqrt{n}}\sum_{k \leq t\sqrt{n} }X_k= \frac{1}{\sqrt{n}}\sum_{k \leq a_n}X_k,
$$
and $L(t)-L(\alpha_n) \ll \frac{1}{\sqrt{n}}.$  For $t$ satisfying $t\sqrt{n}\geq n$ (that is, $t\geq \sqrt{n}$), we have
$$
\frac{1}{\sqrt{n}}\sum_{k \leq t\sqrt{n} }X_k= \frac{1}{\sqrt{n}}\sum_{k \leq n}X_k,
$$
and $L(t)-L(\sqrt{n}) \ll e^{-\frac{\sqrt{n}}{A}}$.  Thus, the infimum in \eqref{E:LimitShapealpha} may be extended over all real $t \geq 0$, and the proposition follows.
\end{proof}

\end{document}